\documentclass[11pt]{amsart}
\usepackage{amsthm}
\usepackage{amsmath}
\usepackage{amsfonts}
\usepackage{amssymb}
\usepackage[all,cmtip]{xy}
\usepackage{cite}
\newtheorem{thm}{Theorem}[section]

\newtheorem{cor}[thm]{Corollary}

\newtheorem{remark}[thm]{Remark}
\newtheorem{lemma}[thm]{Lemma}

\newtheorem{ex}[thm]{Example}
\newtheorem{defn}[thm]{Definition}

\newcommand{\bb}[1]{\mathbb{#1}}
\newcommand{\cl}[1]{\mathcal{#1}}

\numberwithin{equation}{section}

\begin{document}
\title[Bohr's Inequality]{Bohr's Inequality for Non-commutative Hardy Spaces}

\author{Sneh Lata}
\address{Department of Mathematics\\
         School of Natural Sciences\\
         Shiv Nadar University\\
         NH-91, Tehsil Dadri
         Gautam Budh Nagar 201314\\
         Uttar Pradesh, India}
\email{sneh.lata@snu.edu.in}

\author{ Dinesh Singh}
\address{SGT University\\
        Gurugram 122505\\
        Haryana, India}
\email{dineshsingh1@gmail.com}
\thanks{First author's research is supported in part by a grant from SERB(DST) under MATRICS Scheme, India.} 

%\date{}
%\commby{}

\begin{abstract} 
%The classical Bohr's inequality states that if 
%$f(z)=\sum\limits_{n=0}^\infty a_nz^n$ is holomorphic and bounded on the open unit disc $\bb D,$ then 
%$\sum\limits_{n=0}^\infty |a_n|r^n \le ||f||_\infty$ for $0\le r\le \frac{1}{3}.$ 
%In recent times the inequality has found important applications and generalizations beginning with Dixon's, \cite{Dix}, solution in the negative to the von Neumann inequality conjecture for non-unital Banach algebras. 
%In this paper we establish Bohr's inequality in the setting of the non-commutative Hardy space $H^1$ associated with a semifinite von Neumann algebra. 
In this paper we extend the classical Bohr's inequality to the setting of the non-commutative Hardy space $H^1$ associated with a semifinite von Neumann algebra. As a consequence, we obtain Bohr's inequality for operators in the von Neumann-Schatten class $\cl C_1$ and square matrices of any finite order.
%The Bohr's inequality stated above follows as a corollary to our inequality. 
Interestingly, we establish that the optimal bound for $r$ in the above mentioned Bohr's inequality concerning von Neumann-Shcatten class is 1/3 whereas it is 1/2 in the case of $2\times 2$ matrices and reduces to $\sqrt{2}-1$ for the case of $3\times 3$ matrices. We also obtain a generalization of our above-mentioned Bohr's inequality for finite matrices where we show that the optimal bound for $r$, unlike above, remains 1/3 for every fixed order $n\times n,\ n\ge 2$.
\end{abstract}

\maketitle

\section{introduction}
Let $H^\infty(\bb D)$ denote the algebra of all bounded analytic functions on 
the open unit disk $\bb D$ in the complex plane $\bb C$. $H^\infty(\bb D)$ is a Banach 
algebra under the supremum norm $\|\cdot\|_\infty.$ The classical 
Bohr's inequality states that if $f(z)=\sum_{k=0}^\infty a_kz^k$ is 
in $H^\infty(\bb D),$ then 
\begin{equation}\label{first}
\sum_{k=0}^\infty |a_k|r^k\le \|f\|_\infty, \quad 0\le r\le \frac{1}{3}
\end{equation}
and the inequality fails when $r>\frac{1}{3}$ in the sense that there are functions in $H^\infty(\bb D)$ for which the inequality is reversed when 
$r>\frac{1}{3}.$ Henceforth, if there exists a positive real number $r_0$ such that an inequality of the form (\ref{first}) holds for every element of a class $\cl C$ for $0\le r\le r_0$ and fails when $r>r_0,$ then we shall say that $r_0$ is an optimal bound for $r$ in the inequality w.r.t. the class $\cl C.$     

This inequality was first proved by Harald Bohr while dealing with a problem 
connected with Dirichlet series and number theory. Actually, Bohr, \cite{Boh}, established this inequality for $0\le r\le \frac{1}{6}.$ The inequality attracted the attention of a veritable who's who of analysis, and they  extended the validity of the inequality for 
$1\le r\le \frac{1}{3}$ and also established the optimality of $\frac{1}{3}.$ 
The galaxy of analysts included M. Riesz \cite{Dix}, Schur \cite{Dix}, Sidon \cite{sid}, Tomic \cite{tom} and Weiner \cite{Boh}. 

Interest in this inequality was revived a few years ago when Dixon, 
\cite{Dix}, used it to settle in the negative the conjecture that a 
non-unital Banach algebra that satisfies the von Neumann 
inequality must be isometrically isomorphic to a closed subalgebra of $B(H)$ 
for some Hilbert space $H.$ In fact, the second author of this paper along 
with his collaborators used Bohr's inequality to establish something vastly more general by proving that every Banach algebra has an equivalent 
norm under which it satisfies the non-unital von Neumann inequality, \cite{PPS}. In the past few years, Bohr's inequality has become the subject of extensive investigation and the inequality has been extended in numerous directions and settings, a sample of which can be gauged from \cite{MA2,AT,BCQ}, \cite{EPR}, \cite{PPS,PS,PS2}, \cite{Pop2}.
%\cite{Muh,MA2,Aiz,AT,BCQ}, \cite{BK}, \cite{DT}, \cite{EPR}, \cite{GM}, 
%\cite{KP,PPS,PS,PS2}, \cite{Pop,Pop2,Pop3}.

In this paper we establish Bohr's inequality in the setting of a non-commutative Hardy space $H^1$ associated with a semifinite von Neumann algebra. As a consequence, we obtain Bohr's inequality for operators in the 
von Neumann-Schatten class $\cl C_1$ and square matrices of any finite order. In addition, we give two more results that generalize these results regarding the class $\cl C_1$ and finite matrices. The operators of $\cl C_1$ are called the trace-class operators.  
   
We must remark that many results on the classical Hardy spaces have been established for non-commutative Hardy spaces associated with finite and semi-finite von Neumann algebras. The interested readers can refer to 
\cite{BL, BL3, BL4}, \cite{Lab2}, \cite{PX}
%\cite{BL, BL2, BL3, BL4}, \cite{BX}, \cite{Lab, Lab2, MW,MW3}, \cite{PX}, \cite{Sag} 
and the references therein. However, we mention that the Bohr's inequality has not been explored yet in this non-commutative setting. 

\section{backdrop - reproducing a generalized version of Bohr's inequality} 
In \cite{PS}, the second author of this paper (with Vern Paulsen) has established a vastly general form of Bohr's inequality in the context of the Hardy spaces associated with uniform algebras. The classical inequality follows as a trivial corollary from this generalization. In this paper we establish various analogues of this generalized version.  

Over here, we state this general inequality after laying the basic definitions and terminology so that it shall bring clarity in establishing the fact that our results are analogues of this far reaching generalization.  

Let $X$ be a compact Hausdorff space and let $C(X)$ be the Banach algebra of all complex-valued continuous functions from $X$ into $\bb C$ under the supremum norm. A closed subalgebra $A$ of $C(X)$ is said to be a uniform algebra if it contains the constants and separates points of $X.$ Let $m$ be a representing measure for a fixed homomorphism from $A$ into $\bb C.$ The Hardy space $H^1(dm)$ associated with the Lebesgue space $L^1(dm)$ is the closure of the uniform algebra $A$ in $L^1(dm).$ In addition, $H^\infty_0(dm)$ stands for the weak-star closure of $A_0$ in the Lebesgue space $L^\infty(dm)$ treated as the dual of $L^1(dm)$, where $A_0$ stands for the kernel of the fixed homomorphism.  

We need the following result from \cite[Page 97]{Gam}: If $h$ is 
any function in $H^\infty_0(dm)$ and $g$ is in $H^1(dm),$ then 
$\int_X hg dm =0.$ We are now in a position to state the following theorem from \cite{PS}.

\begin{thm}(Bohr's Inequality for Uniform Algebras)\label{unialg} Let $f$ be in $H^1(dm)$ such that $Re f\le 1$ and 
$\int_X f dm\ge 0.$ Let $\psi_0=1$ and $\{\psi_n\}_{n=1}^\infty$ be a sequence in $H_0^\infty(dm)$ such that $||\psi_n||_\infty\le 1.$ Let 
$a_n= \int_X f\overline{\psi_n}dm.$ Then 
$$
\sum_{n=0}^\infty |a_n|r^n \le 1
$$
for all $0\le r\le \frac{1}{3}.$
\end{thm} 

\begin{cor}\label{bineq}(Bohr's Inequality) If $f(z)=\sum_{n=0}^\infty a_nz^n$ is in 
$H^\infty(\bb D),$ then $\sum_{n=0}^\infty |a_n|r^n\le ||f||_\infty$ for all 
$0\le r\le \frac{1}{3}.$
\end{cor}
\begin{proof} Observe that there is no loss of generality 
in assuming in the classical case that $||f||_\infty \le 1$ and 
$f(0)\ge 0.$ Let $A$ be the set of all functions in $C(\bb T)$ whose Fourier series are of analytic type. Then $A$ is a uniform algebra, and normalized Lebesgue measure is a representing measure of the complex-valued homorphism 
$f\mapsto \int_{\bb T}f dm$ on $A.$ In this case, $H^1(dm)$ and $H^\infty(dm)$ are the classical Hardy space 
$H^1(\bb T)$ and $H^{\infty}(\bb T),$ respectively. Further, $H^{\infty}_0(dm)$ turns out to be simply the set of all elements of 
$H^{\infty}(\bb T)$ such that $\int_{\bb T} f dm=0.$ The result now follows from Theorem \ref{unialg} by using the 
standard identification of $H^\infty(\bb D)$ with $H^\infty(\bb T)$ and taking $\psi_n(e^{i\theta})=e^{in\theta}.$
   
\end{proof}

\section{Prelude to Our Results} In this paper we derive analogues of Theorem \ref{unialg} in the setting of Hardy spaces associated with semi-finite von Neumann algebras. Needless to say, the classical Bohr's inequality 
follows as a corollary of one of our results. We also produce analogues of the classical Bohr's inequality in the setting of operators in the von Neumann-Schatten class $\cl C_1$ and finite matrices.  

The organization of the rest of the paper is as follows. In Section \ref{vN}, we first lay down basic definitions and facts about non-commutative Hardy spaces. We then prove our main theorem (Theorem \ref{MT2}) which is an analogue of Theorem \ref{unialg} in the setting of semifinite von Neumann algebras.  As a consequence we obtain Bohr's inequality for trace-class operators and finite matrices for any fixed order. In this same result, we show that when $\cl M$ equals $B(H)$ for an infinite dimensional separable 
Hilbert space $H$ or it belongs to the class of commutative von Neumann algebras, then the optimal bound for $r$ comes out to be $1/3$ just like in the classical case (Corollary \ref{bineq}) as well as its generalization Theorem \ref{unialg}. Further, we show that the situation is different with finite matrices. For this we prove that when 
$\cl M=M_n(\bb C)$ in Theorem \ref{MT2}, then the optimal bound for $r$ is $1/2$ for $n=2$ and it reduces to 
$\sqrt{2}-1$ for $n=3.$

Lastly, as Remarks \ref{Sch2} and \ref{MT1}, respectively, we give generalizations of these above-mentioned Bohr's inequalities for trace-class operators and finite matrices. The optimal bound for $r$ for the generalization written in Remark \ref{Sch2} comes out to be 1/3 which is same as the one for the trace-class operator case of Theorem \ref{MT2}. But the result documented as Remark \ref{MT1} is a little surprising in the sense that, unlike the finite matrices case of Theorem \ref{MT2}, the optimal bound for $r$ here  turns out to be $1/3$ for all $n\ge 2.$

\section{Bohr's inequality for Semifinite von Neumann algebras}\label{vN}
In this section we will give Bohr's inequality (Theorem \ref{MT2}) for Hardy spaces  associated with a semifinite von Neumann algebra  which is essentially an analogue of Theorem \ref{unialg} for semifinite von Neumann algebras. For this, we first perform a brief recap of non-commutative Hardy spaces associated with a semifinite von Neumann algebra.

\begin{defn} Let $\cl M$ be a von Neumann algebra and $\cl M_+$ be its positive part. A function $\tau:\cl M_+\to [0,\infty]$ is said to be a trace on $\cl M$ if it satisfies the following conditions:
\begin{enumerate}
\item $\tau(x+y)=\tau(x)+\tau(y)$ for all $x,y\in \cl M_+,$
\item $\tau(\lambda x)=\lambda \tau(x)$ for all $x\in \cl M_+$ and $\lambda\ge 0,$
\item $\tau(x^*x)=\tau(xx^*)$ for all $x\in \cl M.$
\end{enumerate}   
\end{defn} 

Here we are following the usual convention $0. \infty=0.$ A trace $\tau$ is said to be {\em faithful} if $\tau(x)=0$ implies $x=0;$ 
{\em semifinite} if for every non-zero $x\in \cl M_+$ there exists a non-zero 
$y\in \cl M_+$ such that $y\le x$ and $\tau(y)<\infty;$ {\em finite} if 
$\tau(1)<\infty,$ and {\em normal} if 
$\tau({\rm sup}x_\alpha)={\rm sup}\tau(x_\alpha)$ for every bounded increasing net 
$\{x_\alpha\}$ in $\cl M_+.$ Note that every finite trace is semifinite. 

\begin{defn} A von Neumann algebra $\cl M$ is said to be a semifinite von Neumann algebra if it admits a normal semifinite faithful (n.s.f) trace.  
\end{defn}

The literature on non-commutative $L^p$ spaces is vast. We refer the 
interested readers to the beautiful survey article \cite{PX} and the 
references therein. For this work, we shall need only non-commutative $L^1$ and $L^\infty$ 
spaces, so we will restrict ourselves to these spaces only.  

Henceforth, $\cl M$ will always denote a fixed semifinite von Neumann algebra with a n.s.f trace $\tau.$ Let $L^1(\cl M)$ be the 
associated non-commutative Lebesgue space $L^1.$ Further, let $L^\infty(\cl M)=\cl M$ equipped with its operator norm.  

The usual H$\ddot{o}$lder's inequality extends to the $L^p(\cl M)$ spaces. For any $x\in L^1(\cl M), \ y\in L^\infty(\cl M),$ we get that $xy$ and $yx$ are in $L^1(\cl M)$ with $\tau(xy)=\tau(yx)$ and 
$$
|\tau(xy)|\le ||xy||_1\le ||x||_1||y||.
$$

Let $\cl D$ be a von Neumann subalgebra of $\cl M$ such that $\tau$ restricted to $\cl D$ is again a n.s.f trace. Let $\cl E$ be the normal faithful conditional expectation of $\cl M$ onto $\cl D$ with respect to $\tau.$ 

\begin{defn} A weak-star closed subalgebra $\cl A$ of $\cl M$ is called a subdiagonal algbera of $\cl M$ with respect to $\cl E$ (or $\cl D$) if 
\begin{enumerate}
\item $\cl A+\cl A^*$ is weak-star dense in $\cl M,$
\item $\cl E(xy)=\cl E(x)\cl E(y), \ {\rm for} \ x, \ y\in \cl A,$
\item $\cl A \cap \cl A^* = \cl D.$
\end{enumerate}
\end{defn}

The von Neumann subalgebra $\cl D$ is called the diagonal of $\cl A.$ Henceforth, we shall denote $\cl A$ by $H^\infty(\cl M)$. It is proved by 
Ji in \cite{Ji} that a subdiagonal algebra $H^\infty(\cl M)$ is automatically maximal. This maximality yields the following useful characterization of 
$H^\infty(\cl M).$ 
\begin{equation}\label{subdiag}
H^\infty(\cl M) = \{x\in \cl M: \cl E(axb)=0 \ \forall \ a\in H^\infty(\cl M), \ b\in H^\infty_0(\cl M) \},
\end{equation}
where $H^\infty_0(\cl M)= H^\infty(\cl M)\cap Ker(\cl E).$

\vspace{.2cm}

The non-commutative Hardy space $H^1(\cl M)$ is defined as the closure of 
$H^\infty(\cl M)\cap L^1(\cl M)$ in the $||\cdot||_1$-norm. Before presenting our analogue of Bohr's inequality in the context of von Neumann algebras, we first list the following examples of semifinite von Neumann algebras, a subdiagonal algebra in them, and the associated $H^1$ and $H^\infty_0$ spaces which we will need in the result to settle our claims regarding the optimal bounds. 

\begin{ex}\label{matrices} Recall that the usual trace $Tr$ on the von Neumann algebra $B(H)$ is a n.s.f trace. Let $\cl M= B(H)$ and let $\tau=Tr.$ Then, $L^1(\cl M)$ is the von Neumann-Schatten class $\cl C_1.$ 
\begin{enumerate}
\item Suppose $H$ is an infinite dimensional separable Hilbert space. Fix an orthonormal basis 
$\{e_n\}_{n=0}^\infty$ for $H$. Let $\cl D$ be the class of diagonal operators in $B(\cl H).$ Then $\cl D$ is a von Neumann subalgebra of 
$\cl M=B(H)$ on which the restriction of the trace is again a n.s.f. trace. 
The conditional expectation $\cl E$ of $\cl M$ onto 
$\cl D$ is simply the projection onto the diagonal, that is, $\cl E(T)=D$ where $D$ is the diagonal of $T.$ Further, $\cl U,$ the class of all upper triangular operators in $B(H)$ is a subdiagonal algebra of $\cl M$ with respect to $\cl E.$ Thus we can take $H^\infty(\cl M)=\cl U.$ Then $H^\infty_0(\cl M)=\{T\in \cl U: \langle{Te_i,e_i}\rangle=0 \ \forall \ i\}$ and $H^1(\cl M)$ is the closure of $\cl U\cap \cl C_1$ in the $||\cdot||_1$-norm. 

\item Suppose $H=\bb C^n$ and $\{e_1,\dots,e_n\}$ is the standard basis for it. Then 
$B(H)$ can be identified with $M_n(\bb C).$ Let $\cl D, \ \cl U,$ and $\cl E$ denote the same classes in $B(\cl H)$ as above. Since $\cl H$ is finite dimensional, therefore 
$L^1(\cl M)=M_n(\bb C), \ H^1(\cl M)=H^{\infty}(\cl M) = \cl U,$ and $H^\infty_0(\cl M)$ is the set of $n\times n$ upper triangular matrices with zero diagonal.     
\end{enumerate} 
\end{ex}

\begin{ex}\label{comm} Let $\cl M$ be a commutative von Neumann algebra. Then $\cl M=L^\infty(\Omega, \mu)$ for a measure space 
$(\Omega, \mu).$ The integral with respect to the measure $\mu$ gives a n.s.f. trace on $\cl M.$ In this case, $L^1(\cl M)$ is simply the classical $L^1(\Omega,\mu).$ Take $\Omega = \bb T, \ \mu = m,$ the Lebesgue measure on the unit circle $\bb T$ and $\tau$ the integral with respect to the Lebesgue measure on $\bb T.$ Let $\cl D$ denote the set of all constant ``functions" in 
$\cl M=L^\infty(\bb T).$ Then $\cl D$ is a von Neumann subalgebra of $\cl M,$ the restriction of $\tau$ to $\cl D$ is a n.s.f trace, and the conditional expectation $\cl E$ of $\cl M$ onto $\cl D$ equals $\tau.$  Further, $H^\infty(\bb T)$ is a subdiagonal algebra of of $\cl M$ with respect to $\cl E$ (or $\cl D$). Thus we can take $H^\infty(\cl M)=H^\infty(\bb T).$ Then $H^1(\cl M)=H^1(\bb T)$ and $H^\infty_0(\cl M)$ equals the set of elements in $H^\infty(\bb T)$ with Lebesgue integral zero. 
\end{ex}

The following is the non-commutative analogue of the result (stated on page 3) from \cite[Page 97]{Gam} which was used in the proof of Theorem \ref{unialg}.
\begin{lemma}\label{ML2} Let $x\in H^1(\cl M)$ and $a\in H^\infty_0(\cl M).$ Then $\tau(xa)=0.$   
\end{lemma}

We are now in a position to present our main result. This is an analogue for non-commutative Hardy spaces associated with semifinite von Neumann algebras of Bohr's inequality as given in Theorem \ref{unialg} .

\begin{thm}\label{MT2} Let $x\in H^1(\cl M)$ such that $\tau(x)\ge 0$ and $Re(x)\le y$ for some self-adjoint $y\in H^1(\cl M)$ with finite trace. For any sequence $\{x_m\}_{m=1}^\infty$ in $H^\infty_0(\cl M)$such that  $||x_m||_\infty\le 1$ for each $m\ge 1,$ 
let $\alpha_0=\tau(x)$ and $\alpha_m=\tau(xx_m^*)$ for $m\ge 1.$  
Then 
\begin{equation}\label{bohr2}
\sum_{m=0}^\infty |\alpha_m|r^m\le \tau(y)
\end{equation}
whenever $0\le r\le \frac{1}{3}.$ Further, 
\begin{enumerate} 
\item[(i)] when $\cl M=M_n(\bb C),$ then the optimal bound for $r$ is 
\begin{enumerate}
\item at most $\frac{n}{3n-2}$ for any general $n,$
\item 1/2 for $n=2$,
\item $\sqrt{2}-1$ for $n=3$;
\end{enumerate}
\item[(ii)] when $\cl M=B(\cl H)$ for some infinite dimensional Hilbert space $\cl H$ or $\cl M$ is a commutative von Neumann algebra, then the optimal bound for $r$ is 1/3. 
\end{enumerate}
\end{thm}

\begin{proof} Clearly, to prove the inequality (\ref{bohr2}) it is enough to prove it for $r=\frac{1}{3}.$ Now for each $m\ge 1, \ x_m\in H_0^\infty(\cl M),$ therefore we can write 
\begin{eqnarray*}
-\alpha_m &=& \tau((y-x)x_m^*)\\
&& (since \ \tau(yx_m^*)=\overline{\tau(yx_m)}=0, \ using \ Lemma \ \ref{ML2})\\
&=& \tau((y-x)x_m^*)+\tau((y-x^*)x_m^*)\\
&& (since \ \tau(x^*x_m^*)=\overline{\tau(xx_m)}=0, \ using \ Lemma \ \ref{ML2})\\
&=& \tau(2Re(y-x)x_m^*)\\
&=&2 \tau((y-Re(x))x_m^*).
\end{eqnarray*}
Thus, for each $m\ge 1$ 
\begin{eqnarray*}
|\alpha_m| &=& 2 |\tau(y-Re(x))x_m^*)|\\
&\le & 2 \tau(|y-Re(x)|)||x_m|| \quad \ (using \ H\ddot{o}lder's \ inequality)\\
&=& 2(\tau(y)-\tau(x)) \quad (since \ \tau(Re(x))=\tau(x) and \ ||x_m||\le 1).
\end{eqnarray*}
Hence $\sum_{m=1}^\infty \frac{|\alpha_m|}{3^m}\le \tau(y)-\tau(x)$ which proves the inequality (\ref{bohr2}). We will now settle our claims regarding the optimal bounds for $r$ in inequality (\ref{bohr2}) for the special choices of $\cl M$ as mentioned in the statement of the theorem.  

\vspace{.2 cm}

\noindent\underline{{\bf Case:} $\cl M=M_n(\bb C)$} Recall from Example \ref{matrices} that for $\cl M=M_n(\bb C),$ we can take $H^\infty(\cl M)$ to be the set of upper triangular matrices in $M_n(\bb C)$. Then $H^1(\cl M)=H^\infty(\cl M)$ and 
$H^\infty_0(\cl M)$ equals the set of upper triangular matrices in $M_n(\bb C)$ with zero diagonal.  

\vspace{.2 cm}

\noindent \underline{{\bf Subcase: General $n$}} Take 
\begin{equation*}
A=\left(
\begin{array}{ccccc}
1 & -2 & -2 & \ldots & -2 \\
0 & 1 & -2 & \ldots & -2 \\
\vdots & \vdots & \vdots & \ddots & \vdots \\
0 & 0 & 0& \ldots & -2\\
0 & 0 & 0& \ldots & 1
\end{array}\right),
 \ A_m=\left(
\begin{array}{cccccc}
0 & 1 & 0 & \ldots & 0 \\
0 & 0 & 1 & \ldots & 0 \\
\vdots & \vdots & \vdots & \ddots & \vdots \\
0 & 0 & 0& \ldots & 1\\
0 & 0 & 0& \ldots & 0
\end{array}\right)
\end{equation*}
for each $m$ and $S=2I_n,$ where $I_n$ is the $n\times n$ identity matrix. 

Then $S-Re(A)$ is the $n\times n$ matrix with 1 in each entry which is a non-negative matrix. Hence $A$, $S$ and $\{A_m\}$ satisfy the hypotheses of our theorem. Further, $\alpha_m=Tr(AA_m^*)=Tr(AB^*)=-2(n-1).$ Thus, 
$$
Tr(A)+\sum_{n=1}^\infty|\alpha_m|r^m=n+\frac{2(n-1)r}{1-r} \le 2n = Tr(S)
$$
if and only if $r\le \frac{n}{3n-2}.$ Thus,
this set of matrices shows the optimal bound for $r$ in the inequality (\ref{bohr2}) can at most be $\frac{n}{3n-2}$ when $\cl M=M_n(\bb C).$ 

\vspace{.2 cm}

\noindent\underline{{\bf Subcase: $n=2$}} We will first prove that for this case inequality (\ref{bohr2}) can be proved for a bigger range for $r,$ namely for $ 0\le r\le 1/2$. Let $A\in M_2(\bb C)$ such that $Tr(A)\ge 0$ and $Re(A)\le S$ for some self-adjoint $S\in M_2(\bb C).$  Further, let $\{A_m\}_{m=1}^\infty$ be a sequence in $M_2(\bb C)$ such that $A_m$ has zero diagonal and $||A_m||\le 1$ for each $m.$   

We write   
$$A=\left[
\begin{array}{cc}
\alpha & 2\delta\\
0&\beta
\end{array}\right], \  
S=\left[
\begin{array}{cc}
\eta & 0\\
0&\nu
\end{array}\right],
 \ and \ 
A_m=\left[
\begin{array}{cc}
0 & \mu_m\\
0&0
\end{array}\right]$$
for some $\alpha, \beta, \delta,\mu_m\in \bb C$ and $\eta,\nu\in \bb R.$  Now since $Tr(A)\ge 0$ and $Re(A)\le S,$ therefore $\alpha+\beta\ge 0, \ 0\le Re(\alpha)\le \eta, \ 0\le Re(\beta)\le \nu$ and  
$|\delta|\le \sqrt{(\eta-Re(\alpha))(\nu-Re(\beta))}$. Further, $||A_m||\le 1$ implies that $|\mu_m|\le 1.$ Then 
$\alpha_m=Tr(AA_m^*)=2\delta\overline{\mu_m}$ which implies that
\begin{eqnarray*}
\sum_{m=1}^\infty |\alpha_m| r^m &=& 2\sum_{m=1}^\infty|\delta \mu_m| r^m\\
&\le & 2|\delta |\frac{r}{1-r} \\
&\le & 2 \sqrt{(\eta-Re(\alpha))(\nu-Re(\beta))}\frac{r}{1-r}\\
&\le & \left[{\eta-Re(\alpha)+\nu-Re(\beta)}\right]\frac{r}{r-1}\\
&& \hspace{2 cm} ({\rm Since} \ 2 \sqrt{xy}\le x+y \ \forall x, y\ge 0)\\
&=& \left[(\eta+\nu)-Tr(A)\right]\frac{r}{1-r}\\
&\le & Tr(S)-Tr(A) \quad {\rm whenever} \ r\le \frac{1}{2}.
\end{eqnarray*}
Hence, 
$$
|\alpha_0|+\sum_{m=1}^\infty |\alpha_m| r^m\le Tr(S)
$$ whenever $r\le \frac{1}{2}.$ Further, note that $\frac{n}{3n-2}$ equals $1/2$ for $n=2.$Therefore the optimality of $r=1/2$ follows by taking $n=2$ in the above subcase where we have shown that the optimal bound for $r$ in Inequality (\ref{bohr2}) can be at most $\frac{n}{3n-2}$ for $\cl M=M_n(\bb C).$ 

\vspace{0.2cm}

\noindent\underline{{\bf Subase: $n=3$}} First we will prove that when $\cl M=M_3(\bb C),$ the inequality (\ref{bohr2}) remains true for $0\le r\le \sqrt{2}-1.$  

Let $A, S$ and $\{A_m\}$ be $3\times 3$ matrices which satisfy the hypotheses of the Theorem \ref{MT2}. Let
\begin{equation*}
A=\left(
\begin{array}{ccc}
a_1& 2x & 2y\\
0 & a_2 & 2z\\
0 & 0 & a_3
\end{array}\right)
{\rm and} \  A_m=\left(
\begin{array}{ccc}
0 & u_m & v_m\\
0 & 0 & w_m\\
0 & 0 & 0
\end{array}\right)
\end{equation*}
Set $P=S-Re(A).$ Since $S$ is diagonal, therefore $p_{12}=-x,\ p_{13}=-y$ and $p_{23}=-z,$ where $p_{ij}$ denote the $(i,j)^{th}$ entry of $P.$ Then 
$|x|\le \sqrt{p_{11}p_{22}}, \ |y|\le \sqrt{p_{11}p_{33}}$ and 
$|z|\le \sqrt{p_{22}p_{33}},$ since $P\ge 0.$ For notational simplicity, we will write $p_{ii}$ simply as $p_i$ for each $i=1,2,3.$
Then
\begin{eqnarray*} 
|\alpha_m|&=&|Tr(AA_m^*)|\\
&=&2|xu_m+yv_m+zw_m|\\
&\le& 2(|x||u_m|+|y||v_m|+|z||w_m|)\nonumber\\
&\le&2(|x||u_m|+|y|\sqrt{(1-|u_m|^2)(1-|w_m|^2)}+|z||w_m|)\\
&\le & 2\left(\sqrt{p_1p_2}|u_m|+\sqrt{p_1p_3}\sqrt{(1-|u_m|^2)(1-|w_m|^2)}+\sqrt{p_2p_3}|w_m|\right)
\end{eqnarray*}

If $p_1=0$, then for each $m\ge 1, \ |\alpha_m|\le 2\sqrt{p_2p_3}\le p_2+p_3,$ since $p_2,p_3\ge 0.$ This will imply that $\sum_{m=1}^\infty|\alpha_m|r^m\le p_2+p_3=Tr(S)-Tr(A)$ for every $0\le r \le 1/2.$ A similar set of arguments will show that if $p_2=0$ or $p_3=0,$ then also the desired inequality holds for $0\le r\le 1/2.$ So, we now focus on the case when $p_1,p_2$ and $p_3$ are all non-zero.  
 
If $u_m=0=w_m,$ then $|\alpha_m|\le 2\sqrt{p_1p_3}\le p_1+p_3 \le p_1+p_2+p_3=Tr(S)-Tr(A)$ as $p_1,p_2,p_3\ge 0.$ This will imply the desired inequality for $0\le r\le 1/2$. Now if $u_m=0$ and $w_m\neq 0,$ then 
$$
|\alpha_m| \le 2\left(\sqrt{p_1p_3}+\sqrt{p_2p_3}\right)\le \frac{2}{\sqrt{2}} (p_1+p_2+p_3)=\sqrt{2}(Tr(S)-Tr(A))
$$ 
since $\sqrt{2}(\sqrt{ac}+\sqrt{bc})\le (a+b+c)$ for non-negative real numbers $a,b,c.$ Similar arguments will yield $|\alpha_m|\le \sqrt{2}(Tr(S)-Tr(A)$ for the case when $w_m=0$ and $u_m\neq 0.$ Therefore, 
 $\sum_{m=0}^\infty|\alpha_m|r^m\le \frac{\sqrt{2}r}{1-r}\left(Tr(S)-Tr(A)\right)\le Tr(S)-Tr(A)$ for $0\le r\le \sqrt{2}-1$ when either $u_m=0$ or $w_m=0.$  

We will now turn to the situation when $u_m, \ w_m$ as well as $p_1, \ p_2, \ p_3$ are all non-zero. Consider,
\begin{eqnarray*}
&&\frac{Tr(S)-Tr(A)}{\sqrt{p_1p_2}|u_m|+\sqrt{p_1p_3}\sqrt{(1-|u_m|^2)(1-|w_m|^2)}+\sqrt{p_2p_3}|w_m|}\\
&=&\frac{p_1+p_2+p_3}{\sqrt{p_1p_2}|u_m|+\sqrt{p_1p_3}\sqrt{(1-|u_m|^2)(1-|w_m|^2)}+\sqrt{p_2p_3}|w_m|}\\
&=&\frac{a^2+b^2+1}{a|u_m|+ab\sqrt{(1-|u_m|^2)(1-|w_m|^2)}+b|w_m|}\\
&\ge & \sqrt{2}
\end{eqnarray*}
by using elementary results from calculus, where $a=\sqrt{p_1/p_2}$ and $b=\sqrt{p_3/p_2}.$ This implies that  
$\sum_{n=1}^\infty|\alpha_m|r^m\le \frac{2}{\sqrt{2}}(Tr(S)-Tr(A))\frac{r}{1-r}\le Tr(S)-Tr(A)$ for $r\le\sqrt{2}-1.$ Hence, 
$$
Tr(A)+\sum_{m=1}^\infty |\alpha_m|r^m \le Tr(S). 
$$
whenever $0\le r\le \sqrt{2}-1.$  

Finally, to prove the optimality of $r=\sqrt{2}-1$ in the above inequality, take 
\begin{equation*}
A=\left(\begin{array}{ccc}
2 & -2\sqrt{2} & -2\\
0 & 2 & -2\sqrt{2}\\
0 & 0 & 2
\end{array}\right),
S=\left(\begin{array}{ccc}
3 & 0 & 0\\
0 & 4 & 0\\
0 & 0 & 3
\end{array}\right)
{\rm and} \ 
A_m=\left(\begin{array}{ccc}
0 & 1 & 0\\
0 & 0 & 1\\
0 & 0 & 0
\end{array}\right) 
\end{equation*}
for each $m.$ Clearly $A, \ S, A_m$ satisfy the hypotheses. Then $|\alpha_m|=4\sqrt{2},$ $Tr(A)=6$ and $Tr(S)=10.$ Therefore, 
$Tr(A)+\sum_{n=1}^\infty|\alpha_m|r^m=6+4\sqrt{2}\frac{r}{1-r}>10
$ 
whenever $r>\sqrt{2}-1.$ This establishes the optimality of $\sqrt{2}-1$ for $\cl M=M_3(\bb C).$

\vspace{.2 cm}

\noindent \underline{{\bf Case:} $\cl M=B(H)$} Note that the sequence $\frac{n}{3n-2}$ converges to $\frac{1}{3}.$ Thus, for each positive number $\epsilon,$ there exists $n$ such that $\frac{n}{3n-2}<\frac{1}{3}+\epsilon.$ Then by {\bf Subcase: General $n$}, we get matrices $A, \ S$ and 
$\{A_m\}$ in $M_n(\bb C)$ which satisfy the hypotheses of the Theorem such that inequality (\ref{bohr2}) fails for 
$r=\frac{1}{3}+\epsilon.$ Lastly, using the fact that $M_n\subseteq B(\cl H)$ (upto unitary) and Example \ref{matrices}, 
we observe that the matrices $A, \ S \in H^1(\cl M) \ A_m\in H^\infty_0(\cl M),$ and they satisfy the hypotheses of the 
Theorem such that inequality (\ref{bohr2}) fails for $r=1/3+\epsilon.$ This establishes the optimality of $r=1/3$ in inequality 
(\ref{bohr2}) for this case. 

\vspace{.2 cm}

\noindent \underline{{\bf Case: $\cl M$ is a commutative von Neumann algebra}}  Take $\cl M=L^\infty(\bb T)$. As discussed in Example \ref{comm} we can take $H^\infty(\cl M)=H^\infty(\bb T).$ Then $H^1(\cl M)=H^1(\bb T)$ and $H^\infty_0(\cl M)$ is the space of all $H^\infty(\bb T)$ functions for which the Lebesgue integral is zero. Now by taking $x\in H^\infty(\bb T), \ y=1,$ and using the identification of $H^\infty(\bb T)$ with $H^\infty(\bb D),$ Theorem \ref{MT2} yields the classical Bohr's inequality as stated in Corollary \ref{bineq}. But for classical Bohr's inequality it is well-known that $r=1/3$ is an optimal bound; hence for the collective class of commutative von Neumann algebras the optimal bound for $r$ in inequality (\ref{bohr2}) is 1/3. 
\end{proof}

\begin{remark} The classical Bohr's Inequality (Corollary \ref{bineq}) follows from Theorem \ref{MT2} by taking $\cl M=L^\infty(\bb T).$ The details have been discussed in the proof of the theorem in the case when $\cl M$ is a commutative von Neumann algebra.
\end{remark}

In view of Example \ref{matrices}, Theorem \ref{MT2} yields a version of Bohr's inequality for finite matrices when $\cl M=M_n(\bb C)$ and for trace-class operators when $\cl M=B(H)$ for some infinite dimensional separable Hilbert space $H.$ Interestingly, the techniques used in the proof of Theorem \ref{MT2} can be used to generalize these versions by obtaining similar inequalities under relaxed hypotheses. We record these generalizations below as remarks.

\begin{remark}\label{Sch2} For the case when $\cl M=B(H)$ for an infinite dimensional separable Hilbert space in Theorem \ref{MT2}, we actually can obtain an inequality similar to inequality (\ref{bohr2}) for a much bigger collection of trace-class operators. Suppose we start with any $A\in\cl C_1$ such that $Tr(A)\ge 0$ and $Re(A)\le S$ for some self-adjoint operator $S\in \cl C_1,$ and suppose there is a sequence $\{A_n\}_{n=1}^\infty$ in $B(\cl H)$ that satisfies $Tr(SA_n)=0, \ Tr(AA_n)=0,$ and $||A_n||\le 1$ 
for each $n\ge 1.$ Then for $\alpha_0=Tr(A)$ and $\alpha_n=Tr(AA_n^*)$ for $n\ge 1,$ we get   
\begin{equation}\label{vNC1}
\sum_{n=0}^\infty |\alpha_n|r^n\le Tr(S)
\end{equation}
whenever $0\le r\le 1/3.$ Note that the finite matrices that were constructed in the proof of Theorem \ref{MT2} to settle the optimality of $r =1/3$ for the case $\cl M=B(H)$ in fact satisfy the hypotheses mentioned above. Hence by using the same set of matrices we obtain optimality of $r=1/3$ in inequality (\ref{vNC1}) as well. 
\end{remark}

\begin{remark}\label{MT1} When $\cl M=M_n(\bb C)$ in Theorem \ref{MT2}, we don't need to restrict ourselves to the class of upper triangular matrices only as is done in Theorem \ref{MT2}. In fact, for any $A\in M_n(\bb C)$ with $Tr(A)\ge 0$ and $Re(A)\le S$ for some self-adjoint $S\in M_n(\bb C)$ suppose we have a sequence $\{A_m\}_{m=1}^\infty$ in $M_n(\bb C)$ such that 
$Tr(SA_m)=0, \ Tr(AA_m)=0,$ and $||A_m||\le 1$ for all $m\ge 1.$ Then  for $\alpha_0=Tr(A)$ and $\alpha_m=Tr(AA_m^*)$ for $m\ge 1,$ we obtain   
\begin{equation}\label{bohr0}
\sum_{m=0}^\infty |\alpha_m|r^m\le Tr(S)
\end{equation} 
whenever $0\le r\le \frac{1}{3}.$ It is interesting to note that, unlike Theorem \ref{MT2} for $\cl M=M_n(\bb C),$ here $r=1/3$ is an optimal bound for $r$ in inequality (\ref{bohr0}) for all $n\ge 2.$
To establish this assertion for $n=2$, given any $r>1/3$ choose $\theta$ such that $\frac{1-r}{2r}<\theta<1$ and choose an integer $k$ such that 
$k>\frac{\theta}{2(1-\theta)}.$ Then taking 
\begin{eqnarray*}
A=\left[
\begin{array}{cc}
\frac{1}{2}+ik & \frac{1}{2}\\
\frac{1}{2} & \frac{1}{2}-ik
\end{array}\right], \quad  
A_m=\left[
\begin{array}{cc}
\frac{-\theta}{2k}&i\theta\\
i\theta &\frac{\theta}{2k}
\end{array}\right].
\end{eqnarray*}
for all $m\ge 1,$ and $S=I_2,$ we can easily check that inequality (\ref{bohr0}) fails for this set of matrices. Hence, $r=1/3$ is the optimal bound for $r$ in inequality (\ref{bohr0}) for $n=2.$ Furthermore, for any $n\ge 3,$ we can use this same set of 
$2\times 2$ matrices to construct $n\times n$ matrices by simply adding extra rows and columns of zeroes to prove that 1/3 is the optimal bound for $r$ in inequality (\ref{bohr0}) for that choice of $n.$ 
\end{remark}

\subsection*{Acknowledgements} Both authors thank the Mathematical Sciences Foundation, Delhi for support and facilities needed to complete the present work.

\end{document}